\documentclass[12pt]{article}
\usepackage{amsmath,amssymb,amsthm,latexsym}
\usepackage{hyperref}

\usepackage{graphicx}
\usepackage{xcolor}
\usepackage{geometry}
\newtheorem{theorem}{Theorem}[section]

\newtheorem{corollary}{Corollary}[section]
\newtheorem{proposition}{Proposition}[section]

\newtheorem{remark}{Remark}[section]

\newtheorem{example}{Example}[section]
\newtheorem{thm}{Theorem}

\title{\textbf{Two classes of Willmore Surfaces in $\mathbb{S}^2\times \mathbb{S}^2$}}
\author {Xiaoling Chai,~Shimpei Kobayashi,~Changping Wang,~Zhenxiao Xie\\}
\date{}
\begin{document}
\maketitle
\begin{abstract}We establish two classification theorems for Willmore surfaces in $\mathbb{S}^2 \times \mathbb{S}^2$. Firstly, we prove that a Willmore surface which is also minimal must be either  a special complex curve given by a slice or a diagonal; or, a minimal surface in a totally geodesic submanifold $\mathbb{S}^2 \times \mathbb{S}^1$ described by a solution of the sinh-Gordon equation in one variable. Secondly, we demonstrate that a Willmore surface is of product type if and only if it is the product of an elastic curve in $\mathbb{S}^2$ and a great circle. 
\end{abstract}


\indent{\bf MSC(2020):\hspace{2mm} 53A30, 53A10, 53C42} 

\section{Introduction}
For a closed surface $x:\Sigma\rightarrow (N^n,g)$ in a Riemannian manifold, the squared $L^2$-norm of the trace-free second fundamental form is a fundamental conformal invariant; that is, it is preserved under conformal changes of the metric. This functional, referred to as the conformally invariant Willmore functional \cite{Michelat,Mondino}, will henceforth be called simply the Willmore functional. The Willmore functional is equivalent, up to a topological invariant, to the integral   
$$\int_{\Sigma}(|\vec{H}|^2+K_{1212})dA,$$
where $\vec{H}$ is the mean curvature vector and $K_{1212}$ denotes the sectional curvature of $(N^n,g)$ restricted to the tangent plane of the surface. In what follows, we will denote this integral by $\mathcal{W}$ and also refer to it as the Willmore functional. This functional exhibits notable links with other fundamental quantities, including, among others, the renormalized area functional within the AdS/CFT correspondence \cite{Alexakis, Graham}. Surfaces that satisfy the Euler-Lagrange equation of this functional are called {\em Willmore surfaces}. 

When the ambient space is a real space form, both the Willmore functional and Willmore surfaces have been extensively investigated. Notable advances include the resolution of the Willmore conjecture in $\mathbb{S}^3$, originally proposed by Willmore in \cite{Willmore} and proved by Marques and Neves in \cite{Marques}.   
Regarding the construction and classification of Willmore surfaces, we refer to \cite{Bryant, Ejiri, WD, MWW} and references  therein. A basic fact is that every minimal surface in a real space form is  automatically Willmore. 

When the ambient space is a non-space form Riemannian manifold, the Euler–Lagrange equation of the Willmore functional has been derived independently by several groups of geometers from different viewpoints; see, for example, \cite{Hu-Li, Mondino, Pedit, Wang-Xie}. Owing to the presence of ambient curvature terms in the Euler–Lagrange equation, not every minimal surface is Willmore in a non-space form Riemannian manifold. 
In \cite{Montiel-Urbano}, Montiel and Urbano proved that the only surfaces in $\mathbb{CP}^2$ that are both minimal and Willmore are the superminimal surfaces of positive spin, i.e., complex curves and minimal Lagrangian surfaces. Recently, the last two authors of  this paper have generalized this result to the self-dual K\"ahler-Einstein surfaces (see Proposition 2.3 in \cite{Wang-Xie}). Note that $\mathbb{S}^2 \times \mathbb{S}^2$ (equipped with the standard product metric and complex structure), being neither self-dual nor anti-self-dual, is another canonical example of a K\"ahler-Einstein surface. This motivates the study of which minimal surfaces in $\mathbb{S}^2\times\mathbb{S}^2$ are Willmore. We refer to a surface as minimal-Willmore if it is both minimal and Willmore.
\begin{thm}\label{thm1}
   Let \( x: \Sigma \rightarrow \mathbb{S}^2 \times \mathbb{S}^2 \) be a minimal-Willmore surface. Then either \( x \) is a special complex curve given by a slice or a diagonal; or, up to an isometry, is contained in a totally geodesic submanifold $\mathbb{S}^2 \times \mathbb{S}^1$, and can be described by a solution of the sinh-Gordon equation in one variable.
\end{thm}
\noindent In other words, when the ambient space is $\mathbb{S}^2\times\mathbb{S}^2$,  within the minimal class the Willmore condition turns out to be  rigid: it forces special 
holomorphicity, or reduces the problem to a totally geodesic hypersurface. 
 Geometrically, the last class of minimal surfaces in Theorem~\ref{thm1} arises from a certain Gauss-map construction applied to a cohomogeneity {one} minimal surface in $\mathbb{S}^3\subset\mathbb{R}^4$, see Remark~\ref{rk-mini}. 

In their work \cite{Urbano}, Tobarro and Urbano gave a beautiful local characterization of minimal surfaces in $\mathbb{S}^2\times\mathbb{S}^2$ that possess no complex points. Their description is formulated in terms of a pair of solutions $v$ and $w$ to the sinh‑Gordon equation 
$$v_{z\bar z}+\frac{1}{2}\sinh (2v)=0.$$
Within this framework, the Willmore condition introduces a new partial differential equation linking $v$ and $w$, see \eqref{eq-mwill}. By analyzing this coupled system, we establish the theorem above. 

Except for seeking Willmore surfaces among minimal surfaces, another natural approach is to use 
the product structure of  $\mathbb{S}^2\times\mathbb{S}^2$ to produce Willmore surfaces. 
\begin{thm}\label{thm2}
   All Willmore surfaces of product type in $\mathbb{S}^2 \times \mathbb{S}^2$ are exhausted by products of an elastic curve in $\mathbb{S}^2$ with a great circle. 
\end{thm}

Note that both of these two classes of examples yield Willmore surfaces that are non-linearly full in $\mathbb{S}^2 \times \mathbb{S}^2$. No linearly full examples are currently known. We point out that constructing explicit examples, or even addressing the existence problem, in non-space-form Riemannian manifolds is particularly challenging; see the efforts in \cite{Modino,Ikoma,Michelat} and the references therein. 

The paper is organized as follows. {Section~\ref{sec2}  begins with the  preliminaries on  $\mathbb{S}^2\times \mathbb{S}^2$ and on surfaces within it. In Section~\ref{sec3}, we first derive the fundamental equations for surfaces in $\mathbb{S}^2\times \mathbb{S}^2$, and subsequently compute the corresponding Willmore equation. Section~\ref{sec4} is devoted to the classification of minimal-Willmore surfaces in $\mathbb{S}^2\times\mathbb{S}^2$. Finally, in Section~\ref{sec5}, we classify Willmore surfaces of product type. 


\section{Preliminaries}\label{sec2}
Let $\mathbb{S}^2$ be the standard 2-sphere, $J$ its complex structure, and $\omega$ its Kähler 2-form, defined by $\omega(\cdot, \cdot) = \langle J \cdot, \cdot \rangle$.

We endow $\mathbb{S}^2\times \mathbb{S}^2$ with the product metric, and also denote it by $\left\langle\;,\;\right\rangle$. On $\mathbb{S}^2\times \mathbb{S}^2$ there are two  complex structures:
\[J_1=(J,J),\quad J_2=(J,-J).\]
Through the natural projection maps $\pi_j$, $j=1,2$ from $\mathbb{S}^2\times \mathbb{S}^2$ to $\mathbb{S}^2$,  the Kähler 2-forms $\omega_j$ ( $j=1,2$), associated with the complex structures $J_j$ are given by:
\[\omega_1=\pi^\ast_1\omega+\pi^\ast_2\omega,\quad \omega_2=\pi^\ast_1\omega-\pi^\ast_2\omega.\]
In this paper, we choose $\pi^\ast_1\omega\wedge\pi^\ast_2\omega$ as the orientation of  $\mathbb{S}^2\times \mathbb{S}^2$.

It is well known that the product manifold $\mathbb{S}^2\times \mathbb{S}^2$ is Einstein  with scalar curvature $4$. Using the complex structures $J_1$ and $J_2$, one can write down the curvature tensor ${R}$ of  $\mathbb{S}^2\times \mathbb{S}^2$ as follows, 
{\small
\begin{equation}\label{R}
	{R}(X,Y,Z,W)=-\frac{1}{2}\big(\langle X,W\rangle\langle Y,Z\rangle-\langle X,Z\rangle\langle Y,W\rangle\nonumber+\langle J_1X,J_2W\rangle\langle J_1Y,J_2Z\rangle-\langle J_1X,J_2Z\rangle\langle J_1Y,J_2W\rangle\big). 
    \end{equation}
}

Let $x:\Sigma\rightarrow \mathbb{S}^2\times \mathbb{S}^2$ be an immersion of an orientable surface $\Sigma$. Considering the two Kähler structures on $\mathbb{S}^2\times \mathbb{S}^2$, one can define two Kähler functions $C_1,C_2:\Sigma\rightarrow\mathbb{R}$ as below,
\[x^\ast\omega_j=C_j\omega_\Sigma,\quad j=1,2,\]
where $\omega_\Sigma$ is the area 2-form of $\Sigma$. Note that $C^2_j\leq1$. 
We say the immersion $x$ is complex (with respect to $J_1$ or $J_2$), if $C^2_1=1$ or  $C^2_2=1$. We say the immersion $x$ is Lagrangian (with respect to $J_1$ or $J_2$), if $C_1=0$ or $C_2=0$. Following \cite{Urbano}, a point $p \in \Sigma$ is called \emph{non-complex} provided that $C_1^2(p) < 1$ and $C_2^2(p) < 1$.   

Associated to the immersion $x$, there are two natural maps 
$x_j\triangleq \pi_j\circ x:\Sigma\rightarrow\mathbb{S}^2$, by which $x=(x_1,x_2)$. 
If the surface $\Sigma$ is compact and oriented, 
then we have 
\begin{align}
	\int_\Sigma C_1\,dA=4\pi(d_1+d_2),\quad\int_\Sigma C_2\,dA=4\pi(d_1-d_2),\label{C_12}
\end{align}
where $d_j$ is the degree of $x_j$. 
Consequently, if $x$ is a complex immersion, then its area is $A=4\pi m,m\in\mathbb{N}$; if $x$ is Lagrangian, then we have $|d_1|=|d_2|$.

	Along the immersion $x$, we choose a local oriented orthonormal frame $\{e_1, e_2, e_3, e_4\}$ for $\mathbb{S}^2 \times \mathbb{S}^2$, such that $\{e_1, e_2\}$ is a tangent frame and $\{e_3, e_4\}$ is a normal frame to $x$. For  the immersion $x$, we denote by $K$ the {Gauss} curvature, $B$ the second fundamental form, $\vec{H}$ the mean curvature vector, $K^\perp = R^\perp(e_1, e_2, e_3, e_4)$ the normal curvature, where $R^\perp$ represents the normal curvature tensor. The Gauss, Coddazzi, and Ricci equations are given by 
	\begin{align}
    &K=\frac{1}{2}(C^2_1+C^2_2)+2|\vec{H}|^2-\frac{|B|^2}{2},\label{eq-K}\\
    &(\nabla B)(X,Y,Z)-(\nabla B)(Y,X,Z)=\frac{1}{2}\big(\left\langle J_1X,J_2Z\right\rangle(J_2J_1Y)^\perp-\left\langle J_1Y,J_2Z\right\rangle(J_2J_1X)^\perp\big),\\
    &K^\perp=\frac{1}{2}(C^2_1-C^2_2)+\langle[A_{e_4},A_{e_3}]e_1,e_2\rangle.
    \end{align}
    where $(\:)^\perp$ denotes the normal component relative to the immersion $x$, and $A_\eta$ represents the shape operator with respect to the normal vector $\eta$. 

\section{The fundamental equations of surfaces in $\mathbb{S}^2\times \mathbb{S}^2$}\label{sec3}

In \cite{Urbano}, the authors established the fundamental equations for minimal surfaces in $\mathbb{S}^2\times\mathbb{S}^2$. In this section we first generalize these equations to arbitrary surfaces, and then employ them to derive the Euler–Lagrange equation of the Willmore functional.    

Let $x= (x_1, x_2):\Sigma \rightarrow \mathbb{S}^2\times \mathbb{S}^2$ be an immersion of a surface $\Sigma$. We choose $z=u+i\,v$ as a local isothermal coordinate and write the induced metric $g$ of $x$ as 
$$g=e^{2\sigma} |dz|^2=e^{2\sigma} (du^2+dv^2).$$
Let $\{e_3, e_4\}$ be an orthonormal frame of the normal bundle such that $\{x_u, x_v, e_3, e_4\}$ forms an oriented frame of $x^* T(\mathbb{S}^2 \times \mathbb{S}^2)$.

Define $\xi \triangleq \frac{1}{\sqrt{2}}(e_4 - i e_3)$. Then $|\xi|^2 = 1$, $\langle \xi, \xi \rangle = 0$, so $\{\xi, \bar{\xi}\}$ is a local orthonormal frame in the complexified normal bundle. According to the given orientation, there exists smooth complex functions $\gamma_1$ and $\gamma_2$ such that 
\begin{align}
&J_1 x_z = i C_1 x_z + \gamma_1 \xi, \label{J1xz} \\
&J_1 \xi = -2 e^{-2\sigma} \bar{\gamma}_1 x_z - i C_1 \xi, \label{J1xi} \\
&J_2 x_z = i C_2 x_z + \gamma_2 \bar{\xi}, \label{J2xz} \\
&J_2 \xi = -2 e^{-2\sigma} \gamma_2 x_{\bar{z}} + i C_2 \xi. \label{J2xi}
\end{align}
It is easy to see that 
\begin{equation}\label{eq-GC}
2|\gamma_j|^2 = {e^{2\sigma}(1 - C_j^2)},~~~1\leq j\leq 2.
\end{equation}

Set $\hat{x} \triangleq (x_1, -x_2)$, then $\{x, \hat{x}, x_z, x_{\bar{z}}, \xi, \bar{\xi}\}$ forms a moving frame along the immersion $x: \Sigma \rightarrow \mathbb{S}^2\times \mathbb{S}^2\rightarrow \mathbb{R}^6$. 
Note that $\hat{x}=-J_1J_2 x$, using \eqref{J1xz} $\sim$ \eqref{J2xi}, we obtain the structure equation of this moving frame as follows, 
\begin{align}
&\hat{x}_z = C_1 C_2 x_z + 2 e^{-2\sigma} \gamma_1 \gamma_2 x_{\bar{z}} - i C_2 \gamma_1 \xi - i C_1 \gamma_2 \bar{\xi}, \label{eq-xh}\\
&x_{zz} = 2 \sigma_z x_z + f_1 \xi + f_2 \bar{\xi} - \frac{\gamma_1 \gamma_2}{2} \hat{x}, \label{F1} \\
&x_{z\bar{z}} = i \frac{\sqrt{2}}{4} e^{2\sigma} \bar H \xi - i\frac{\sqrt{2}}{4} e^{2\sigma} {H} \bar{\xi} - \frac{1}{4} e^{2\sigma} x - \frac{1}{4} e^{2\sigma} C_1 C_2 \hat{x}, \label{F2} \\
&\xi_z = i\frac{\sqrt{2}}{2} {H} x_z - 2 e^{-2\sigma} f_2 x_{\bar{z}} + \rho \xi + \frac{i C_1 \gamma_2}{2} \hat{x}, \label{F3} \\
&\xi_{\bar{z}} = -2 e^{-2\sigma} \bar{f}_1 x_z + i\frac{\sqrt{2}}{2} {H} x_{\bar{z}} - \bar{\rho} \xi - \frac{i C_2 \bar{\gamma}_1}{2} \hat{x}, \label{F4}
\end{align}
where $\rho$ characterizes the normal connection, $i(\bar{H}\xi-{H}\bar{\xi})/\sqrt{2}$ is the mean curvature vector, and 
$(f_1\xi+f_2\bar{\xi})dz^2$ represents the normal bundle valued Hopf form of $x$.  As in \cite{Urbano}, we call 
\[
G \stackrel{\triangle}{=} \{\sigma, \rho, H, f_1, f_2, C_1, C_2, \gamma_1, \gamma_2\} 
\]
the fundamental data of $(x, \xi)$. 

Performing a direct calculation by differentiating \eqref{J1xz}$\sim$\eqref{F4} with respect to $z$ and $\bar{z}$, we obtain the following integrable equations: 
\begin{align}
&(C_1)_z = 2 i e^{-2\sigma} f_1 \bar{\gamma}_1 - \frac{\sqrt{2}}{2}  \gamma_1 {H}, \label{eq-C1z}\\
&(C_2)_z = 2 i e^{-2\sigma} f_2 \bar{\gamma}_2 + \frac{\sqrt{2}}{2}  \gamma_2 \bar{H}, \label{eq-C2z}\\
&(\gamma_1)_{\bar{z}} = \bar{\rho} \gamma_1 + \frac{\sqrt{2}}{2} e^{2\sigma} C_1 \bar{H}, \label{eq-g1z}\\
&(\gamma_2)_{\bar{z}} = -\bar{\rho} \gamma_2 - \frac{\sqrt{2}}{2}  e^{2\sigma} C_2 {H}, \label{eq-g2z}\\
&(f_1)_{\bar{z}} =   \bar{\rho}f_1 + \frac{i\sqrt{2} e^{2\sigma}}{4}\big(\bar{H}_{z}+\rho \bar H \big)+ \frac{1}{4} i e^{2\sigma} C_1 \gamma_1, \label{eq-f1z}\\
&(f_2)_{\bar{z}} = - \bar{\rho}f_2 - \frac{i\sqrt{2} e^{2\sigma}}{4}\big({H}_{z}-\rho H \big)+\frac{1}{4} i e^{2\sigma} C_2 \gamma_2,\label{eq-f2z}\\
&|f_j|^2 = \frac{1}{8} e^{4\sigma} \left( |H|^2 + C_j^2 - K + (-1)^j K^\perp \right), ~~1\leq j\leq 2. \label{eq-fj}
\end{align}

\begin{remark}
    In \cite{Urbano}, the authors showed that if $x$ is complex with respect to $J_1$ (resp. $J_2$) then 
    $C_1^2=1$ and $\gamma_1=f_1=0$ (resp. $C_2^2=1$ and $\gamma_2=f_2=0$).
\end{remark}
\begin{remark}\label{rk-mini2}
    The functions $f_1$ and $f_2$ are exactly the local conformal invariants $i\Psi/\sqrt{2}$ and $-i\Phi/\sqrt{2}$ introduced in \cite{Wang-Xie}. It follows from \eqref{eq-fj} that the Willmore functional $\mathcal{W}(x)$, the functionals $\mathcal{W}^+(x)$ and $\mathcal{W}^-(x)$ (defined by Montiel and Urbano in \cite{Montiel-Urbano}) of the immersion $x$ can be expressed as follows, 
$$\mathcal{W}(x)=\int_{\Sigma}(2|\vec{H}|^2+C_1^2+C_2^2)\,dA_x-4\pi\,\chi(\Sigma),$$
$$\mathcal{W}^+(x)=\int_{\Sigma}(|\vec{H}|^2+C_1^2)\,dA_x,~~~\mathcal{W}^-(x)=\int_{\Sigma}(|\vec{H}|^2+C_2^2)\,dA_x.$$
Consequently, in $\mathbb{S}^2\times\mathbb{S}^2$, complex curves and minimal Lagrangian surfaces with respect to $J_1$ (resp. $J_2$) are minimizers of the conformal invariant functional $\mathcal{W}^+$ (resp. $\mathcal{W}^-$).  
\end{remark}

For closed surfaces in a conformal 4-manifold, the last two authors of this paper calculated the first variation for the Willmore functional in \cite{Wang-Xie}. According to the results in that paper, the Euler-Lagrange equation of the Willmore functional in $\mathbb{S}^2\times \mathbb{S}^2$ is given by 
\begin{multline} \label{eq-Willmore}
2\big[ (f_1)_{\bar{z}\bar{z}} - 2(\sigma_{\bar{z}} + \bar{\rho}) (f_1)_{\bar{z}} + (\bar{\rho}^2 + 2\sigma_{\bar{z}}\bar{\rho} - \bar{\rho}_{\bar{z}}) f_1 \big] 
+ 2\big[ (\bar{f}_2)_{zz} - 2(\sigma_z - \rho) (\bar{f}_2)_z + (\rho^2 - 2\sigma_z\rho + \rho_z) \bar{f}_2 \big] \\
+ \gamma_1^2 \bar{f}_1 + \bar{\gamma_2}^2f_2 
+ i\sqrt{2}\big[ (|f_1|^2+|f_2|^2) \bar{H} - 2 f_1\bar{f}_2 H\big]=0
\end{multline}
We call a surface {\em Willmore} if it satisfies the above Euler-Lagrange equation. In terms of the mean curvature vector, this equation can also be expressed as 
\begin{equation}\label{eq-WH}
\begin{split}
\!\!\!\!\!\bar{H}_{z\bar{z}}+\bar{H}_{\bar{z}z}+2(\rho \bar H_{\bar{z}}- \bar{\rho} \bar H_z) +\!\! \left(\! \frac{3e^{2\sigma}}{4} \big(C_1^2+C_2^2-\frac23\big) - 2|\rho|^2 + \rho_{\bar{z}} - \bar\rho_{{z}}+2 e^{-2\sigma}\big(|f_1|^2+|f_2|^2\big) \!\right)\! \bar{H} \\
-4 e^{-2\sigma} f_1 \bar{f}_2{H} -2\sqrt{2}i e^{-2\sigma} \big(\bar{f}_1 \gamma_1^2+f_2\bar{\gamma_2}^2\big)=0
\end{split}
\end{equation}

Similarly, a surface is a critical point of the functional $\mathcal{W}^+$ if and only if
\[
2\big[ (f_1)_{\bar{z}\bar{z}} - 2(\sigma_{\bar{z}} + \bar{\rho}) (f_1)_{\bar{z}} + (\bar{\rho}^2 + 2\sigma_{\bar{z}}\bar{\rho} - \bar{\rho}_{\bar{z}}) f_1 \big] + \gamma_1^2 \bar{f}_1 -i\sqrt{2} \big( H \bar{f}_2 f_1 - \bar{H} |f_1|^2 \big)=0,
\]
or equivalently, 
\begin{multline}\label{eq-W+}
\bar{H}_{z\bar{z}} - \bar{\rho} \bar{H}_z + \rho \bar{H}_{\bar{z}} + \left( \frac{1}{2} e^{2\sigma} C_1^2 - \frac{1}{2} |\gamma_1|^2 - |\rho|^2 + \rho_{\bar{z}} + 2 e^{-2\sigma} |f_1|^2 \right) \bar{H} - 2 e^{-2\sigma} f_1\bar{f}_2 H 
- 2\sqrt{2}ie^{-2\sigma} \bar{f}_1 \gamma_1^2 = 0
\end{multline}
Likewise, it is a critical point of $\mathcal{W}^-$ if and only if
\[
2\big[ (\bar{f}_2)_{zz} - 2(\sigma_z - \rho) (\bar{f}_2)_z + (\rho^2 - 2\sigma_z\rho + \rho_z) \bar{f}_2 \big] + \bar{\gamma_2^2} f_2-i\sqrt{2} \big( H f_1 \bar{f}_2 - \bar{H} |f_2|^2\big)=0,
\]
or equivalently, 
\begin{multline}
\bar{H}_{\bar{z}z} + {\rho} \bar{H}_{\bar z} - \bar\rho \bar H_{{z}} + \left(\frac12 e^{2\sigma} C_2^2- \frac12 |\gamma_2|^2 - |\rho|^2 -\bar\rho_{{z}} + 2e^{-2\sigma} |f_2|^2  \right) \bar{H} - 2e^{-2\sigma} {f}_1\bar{f_2}{H} - 2\sqrt{2}i e^{-2\sigma} {f}_2 \bar{\gamma_2}^2 = 0
\end{multline}
\begin{proposition}
    Let $x:\Sigma \rightarrow \mathbb{S}^2\times \mathbb{S}^2$ be a minimal surface. Then $x$ is a critical surface of the functional $\mathcal{W}^+$ (resp. $\mathcal{W}^-$) if and only if it is complex or Lagrangian with respect to the complex structure $J_1$ (resp. $J_2$). 
\end{proposition}
\begin{proof}
    It follows from \eqref{eq-W+} that $x$ is a critical surface of $\mathcal{W}^+$ if and only if $\bar{f_2}\gamma_2^2 = 0$. By \eqref{eq-C2z}, we deduce that $C_2$ is a holomorphic function on $\Sigma$ and hence is constant. 
    
    If $C_2^2 \equiv 1$, then $x$ is a complex curve. Otherwise, from \eqref{eq-GC} and \eqref{eq-g2z}, we conclude that $\gamma_2$ is nowhere vanishing, which forces $f_2 \equiv 0$. Then \eqref{eq-f2z} implies $C_2 \equiv 0$, i.e., $x$ is a minimal Lagrangian surface.
    
    A similar argument applies to the case of $\mathcal{W}^-$.
\end{proof}
\begin{remark}
    This proposition implies that, among minimal surfaces in $\mathbb{S}^2 \times \mathbb{S}^2$, the critical surfaces of the conformal invariant functionals $\mathcal{W}^+$ and $\mathcal{W}^-$ are in fact minimizers. A natural question is whether this conclusion extends to the full Willmore functional $\mathcal{W}$.  It follows from \eqref{eq-WH} that minimal-Willmore surfaces are characterized by the equation 
\begin{equation}\label{eq-mW}
f_1 \bar{\gamma}_1^2+\bar{f}_2 \gamma_2^2=0.\end{equation}
\end{remark}
\section{Classification of minimal-Willmore surfaces in $\mathbb{S}^2\times \mathbb{S}^2$}\label{sec4}

In this section, we first discuss several well-known examples of minimal surfaces, and then establish the classification of minimal-Willmore surfaces in  $\mathbb{S}^2\times\mathbb{S}^2$.  

\begin{example}\label{ex-slice} For any point $p\in\mathbb{S}^2$, the corresponding slices 
\[
\mathbb{S}^2\times\{p\}=\{(x,p)\in\mathbb{S}^2\times\mathbb{S}^2 \mid x\in\mathbb{S}^2\},
\]
\[
\{p\}\times\mathbb{S}^2=\{(p,x)\in\mathbb{S}^2\times\mathbb{S}^2 \mid x\in\mathbb{S}^2\}.
\]
are totally geodesic with the first one satisfying $C_1=C_2=1$ and the second one satisfying $C_1=-C_2=1$. That is, the slices are complex with respect to both $J_1$ and $J_2$.
\end{example}
\begin{example}\label{ex-diagonal}  The diagonal 
$$D\triangleq\{(x,x)\in\mathbb{S}^2\times\mathbb{S}^2 \mid x\in\mathbb{S}^2\}$$
is totally geodesic 
and satisfies $C_1=1$ and $C_2=0$, which means that $D$ is complex with respect to $J_1$ and Lagrangian with respect to $J_2$.
\end{example}
\begin{example}\label{ex-Clifford}
    The torus defined by 
    $$T\triangleq\{(x,y)\in\mathbb{S}^2\times\mathbb{S}^2 \mid x_1=y_1=0\}$$ 
    is totally geodesic and satisfies $C_1=C_2=0$, which means that $T$ is Lagrangian for both $J_1$ and $J_2$. This surface is known as the Clifford torus in $\mathbb{S}^2 \times \mathbb{S}^2$, characterized by the flatness of both its tangent and normal bundles.  
\end{example}
\begin{example} 
Let $\Sigma=\mathbb{C}/\Lambda$ be a torus generated by the lattice $\Lambda=\{m+n\tau \mid m,n\in\mathbb{Z}\}$, where $\tau$ is a complex number with $\operatorname{Im}\tau>0$. Let $\wp:\Sigma\rightarrow\mathbb{S}^2$ be the Weierstrass $\wp$-function with a second-order pole at the origin, and $z_0\in\Sigma$ be a point at which $\wp$ is not ramified. 
Then $x=\big(\wp(z),\wp(z-z_0)\big):\Sigma\rightarrow(\mathbb{S}^2\times\mathbb{S}^2)$ is a holomorphic embedding with the complex structure $J_1$. Such minimal tori are referred to as Weierstrass tori.
\end{example}

By \cite[Proposition 3]{Urbano}, the first three examples exhaust all totally geodesic surfaces in $\mathbb{S}^2\times\mathbb{S}^2$, which are trivially Willmore. Among these, the diagonal is the only full surface. In this paper, we will prove that it is, in fact, the only full Willmore surface among all minimal surfaces in $\mathbb{S}^2\times\mathbb{S}^2$.

We first illustrate that the Weierstrass tori are not Willmore by a simple global argument. Note that $C_1^2 \equiv 1$ implies $\gamma_1 \equiv 0$. From \eqref{eq-C2z} and the condition $H \equiv 0$, it follows that $\gamma_2$ satisfies a linear partial differential equation. Using this and \eqref{eq-GC}, together with 
$$\int_\Sigma C_2  dA = 4\pi(d_1-d_2)=0,$$ 
we deduce that $\gamma_2$ is nowhere vanishing; here we have used the fact both $\wp(z)$ and $\wp(z-z_0)$ have mapping degree 2.  If a Weierstrass torus is minimal-Willmore, then \eqref{eq-mW} implies  $f_2\equiv0$. Combining this with \eqref{eq-fj}, we obtain the following  contradiction,  
$$\int_\Sigma C_2^2  dA = \int_\Sigma K  dA-\int_\Sigma K^\perp  dA  = -2\pi\chi^\perp = -16\pi,$$
where the last equality follows from Proposition 2.4 in \cite{Castro-Urbano}. 

\begin{proposition}\label{pro-mw}
    Let $x:\Sigma \rightarrow \mathbb{S}^2\times \mathbb{S}^2$ be a complex curve or minimal Lagrangian surface. Then $x$ is Willmore if and only if it is one of the surfaces given in Example~\ref{ex-slice} $\sim$ Example~\ref{ex-Clifford}. 
\end{proposition}
\begin{proof}
Without loss of generality, we assume that $x$ is complex or minimal Lagrangian with respect to the complex structure $J_1$. Consequently, $C_1$ is a constant which must be either $0$ or $\pm 1$. It follows from \cite[Proposition 3]{Urbano} (see also \cite{Chen}) that we only need to prove that $C_2$ is a constant. 

By \eqref{eq-C1z}, we have $f_1\bar{\gamma_1}=0$. Substituting this into \eqref{eq-mW} yields $\bar{f_1}{\gamma_2}=0$, which, by \eqref{eq-C2z}, implies that $C_2$ is holomorphic. Since $C_2$ is real-valued, it must be constant. 
\end{proof}
\begin{theorem}\label{thm-mw}
    Let $x:\Sigma \rightarrow \mathbb{S}^2\times \mathbb{S}^2$ be a minimal surface without complex points. Then $x$ is Willmore if and only if up to an isometry it is a minimal-Willmore surface in $\mathbb{S}^2\times\mathbb{S}^1$. 
\end{theorem}

To prove this theorem, we need the following characterization of minimal surfaces in $\mathbb{S}^2\times\mathbb{S}^2$ by solutions of the \textup{sinh-Gordon} equation described in Section 4 of \cite{Urbano}. 

 Let $x:\Sigma \rightarrow \mathbb{S}^2\times \mathbb{S}^2$ be a minimal immersion of a simply connected surface without complex points and $z$ be a complex coordinate such that the Hopf differential $\Theta(z)\triangleq\frac{\langle J_1 x_z, J_2 x_z\rangle}{2}dz^2$ (which is holomorphic) satisfies $\Theta(z)=dz^2$. 
 Then the functions 
 $$v\triangleq\frac{1}{2}\ln \sqrt{\frac{(1+C_1)(1+C_2)}{(1-C_1)(1-C_2)}},~~~w\triangleq\frac{1}{2}\ln \sqrt{\frac{(1-C_1)(1+C_2)}{(1+C_1)(1-C_2)}}$$
 satisfy the sinh-Gordon equation, i.e., 
 $$v_{z\bar z}+\frac{1}{2}\sinh (2v)=0,~~~w_{z\bar z}+\frac{1}{2}\sinh (2w)=0.$$
 Conversely, given two solutions $v, w:\mathbb{C}\rightarrow \mathbb{R}$ of the sinh-Gordon equation, one can construct a 1-parameter family of minimal immersions $X_t: \mathbb{S}^2 \times \mathbb{S}^2$ by taking the following quantities as the fundamental data, 
$$\sigma=\frac{1}{2}\ln\big(4\cosh(v+w)\cosh(v-w)\big),~~~ \rho=\left(\ln \sqrt{\frac{\cosh(v+w)}{\cosh(v-w)}}\right)_{\!\!z},~~~$$
$$C_1=\tanh(v-w),~~~C_2=\tanh(v+w),$$
$$\gamma_1=\sqrt{2}e^{\frac{it}{2}}\sqrt{\frac{\cosh(v+w)}{\cosh(v-w)}},~~~\gamma_2=\sqrt{2}e^{\frac{it}{2}}\sqrt{\frac{\cosh(v-w)}{\cosh(v+w)}},~~~f_1=-i \gamma_1 (v-w)_z,~~~f_2=-i\gamma_2(v+w)_z.$$
Geometrically, this construction originates from the Gauss maps of two minimal surfaces in \( \mathbb{S}^3 \), whose induced metrics are \( e^{2v} |dz|^2 \) and \( e^{2w} |dw|^2 \), respectively, and which share the same Hopf differential \( \theta(z) = \frac{i}{2} dz^2 \). For further details, we refer to Section~5 of \cite{Urbano}.

\begin{proposition}
    The minimal immersion $X_t: \mathbb{C} \rightarrow \mathbb{S}^2 \times \mathbb{S}^2$ is Willmore if and only if 
    \begin{equation}\label{eq-mwill}
    e^{it}\cosh^3(v-w)(v+w)_{\bar{z}} = \cosh^3(v+w)(v-w)_z.
    \end{equation}
\end{proposition}
\begin{proof} The conclusion follows directly by substituting the above fundamental data into \eqref{eq-mW}.
\end{proof}

\begin{proof}[{Proof of Theorem~4.1}]
Note that, due to the analyticity of minimal surfaces, it suffices to prove the theorem locally. 
We may therefore assume that $\Sigma = \mathbb{C}$ with a complex coordinate $z = u_1 + i u_2$, so that the immersion $x$ can be described by a pair of solutions $v$ and $w$ of the sinh-Gordon equation. The  minimal-Willmore condition on $x$ is then equivalent to the following system of partial differential equations, 
    \begin{align}
       &v_{z\bar z}=-\frac{1}{2}\sinh (2v), \label{eq-vl}\\
       &w_{z\bar z}=-\frac{1}{2}\sinh (2w), \label{eq-wl}\\
       &
       \cosh^3(v-w)(v+w)_{\bar{z}} = \cosh^3(v+w)(v-w)_z. \label{eq-vw}
    \end{align}

    Set $p \triangleq v + w$ and $q \triangleq v - w$, then the above three equations become 
    \begin{align}
       &p_{z\bar z}=-\sinh (p)\cosh(q),\label{eq-pl}\\
       &q_{z\bar z}=-\sinh (q)\cosh(p),\label{eq-ql}\\
       &
       \frac{p_{\bar z}}{\cosh^3(p)} = \frac{q_{z}}{\cosh^3(q)}.\label{eq-pq} 
    \end{align}
    
    By introducing a real function $H(s)\triangleq\int_{0}^s\frac{dt}{\cosh^3t}$, we obtain that $\eqref{eq-pq}$ is equivalent to 
    \begin{equation}
    \frac{\partial}{\partial \bar{z}} H(p)=\frac{\partial}{\partial {z}} H(q), 
    \end{equation}
    which implies that 
    $$ \frac{\partial}{\partial u_1} \big(H(p)-H(q)\big)=\frac{\partial}{\partial u_2} \big(H(p)+H(q)\big)=0.$$ Therefore, there exists real functions $f(u_1)$ and $g(u_2)$ such that 
    $$H(p)=f(u_1)+g(u_2),~~~H(q)=f(u_1)-g(u_2).$$ 
    By taking partial derivatives to both hand sides of the above two equations, we have 
   \begin{alignat}{2}
&p_{u_1} = \frac{f'(u_1)}{H'(p)} = f'(u_1)\cosh^3(p), \quad\quad\quad\quad\quad  && p_{u_2} = \frac{g'(u_2)}{H'(p)} = g'(u_2)\cosh^3(p), \label{eq-p1}\\
&q_{u_1} = \frac{f'(u_1)}{H'(q)} = f'(u_1)\cosh^3(q), \quad\quad\quad\quad\quad && q_{u_2} = -\frac{g'(u_2)}{H'(q)} = -g'(u_2)\cosh^3(q),\label{eq-q1}\\
&p_{u_1 u_2} = 3p_{u_1} p_{u_2} \tanh(p), \quad\quad\quad\quad\quad && q_{u_1 u_2} = 3q_{u_1} q_{u_2} \tanh(q).\label{eq-pq12}
\end{alignat}

    It follows from \eqref{eq-p1} and \eqref{eq-q1} that 
    $$p_{u_1u_1} + p_{u_2u_2} = \big( f''(u_1) + g''(u_2) \big) \cosh^3(p) + 3 \big( f'(u_1)^2 + g'(u_2)^2 \big) \cosh^5(p) \sinh(p),$$
    $$q_{u_1u_1} + q_{u_2u_2} = \big( f''(u_1) - g''(u_2) \big) \cosh^3(q) + 3 \big( f'(u_1)^2 + g'(u_2)^2 \big) \cosh^5(q) \sinh(q).$$
    Combining these with \eqref{eq-pl} and \eqref{eq-ql}, we obtain 
    \begin{equation}\label{eq-pqs}
    \begin{split}
       \big(f''(u_1)+g''(u_2)\big)\cosh^3(p) + 3 \big( f'(u_1)^2 + g'(u_2)^2 \big) \cosh^5(p) \sinh(p)=-4\sinh(p)\cosh(q), \\
        \big(f''(u_1)- g''(u_2)\big)\cosh^3(q) + 3 \big( f'(u_1)^2 + g'(u_2)^2 \big) \cosh^5(q) \sinh(q)=-4\sinh(q)\cosh(p).
        \end{split}
    \end{equation}
    
   \textbf{Claim.} \emph{There holds $p = \pm q$.}
    \vskip 0.2cm
    To prove this claim, we argue by contradiction. Assume, to the contrary, that $p \neq \pm q$. 

    If $f'(u_1)=g'(u_2)=0$, then \eqref{eq-p1} and \eqref{eq-q1} implies that both $p$ and $q$ are constant. It follows from \eqref{eq-pl} and \eqref{eq-ql} that 
    $p=q=0.$ The contradiction with the hypothesis $p \neq \pm q$ is thus established. 

    For the case $f'(u_1)g'(u_2)\neq0$, we first regard \eqref{eq-pqs} as the following system of linear equations, 
     \begin{equation}\label{eq-fgl1}
    \begin{split}
       \!\!\!\!\! f''(u_1)\cosh^3(p) + 3 \big( f'(u_1)^2 + g'(u_2)^2 \big) \cosh^5(p) \sinh(p)=-4\sinh(p)\cosh(q)- g''(u_2)\cosh^3(p), \\
        \!\!\!\!\! f''(u_1) \cosh^3(q) + 3 \big( f'(u_1)^2 + g'(u_2)^2 \big) \cosh^5(q) \sinh(q)=-4\sinh(q)\cosh(p)+g''(u_2)\cosh^3(q).
        \end{split}
    \end{equation}
    Its coefficient matrix has a non-vanishing determinant given by   $$3\cosh^3(p)\cosh^3(q)\big(\cosh^2(q) \sinh(q)-\cosh^2(p) \sinh(p)\big),$$
    under the hypothesis that $p \neq q$. 
    So we can solve from \eqref{eq-fgl1} that 
    \begin{equation}\label{eq-f''}
    f''(u_1)=A+B\,g''(u_2),
    \end{equation}
    with 
    $$A=A(p,q)=-4 \frac{\sinh(p)\sinh(q)}{\cosh^3(p)\cosh^3(q)} \cdot \frac{\cosh^6(q) - \cosh^6(p)}{\cosh^2(q)\sinh(q) - \cosh^2(p)\sinh(p)},$$
    and 
    $$B=B(p,q)=-\frac{\cosh^2(q)\sinh(q) + \cosh^2(p)\sinh(p)}{\cosh^2(q)\sinh(q) - \cosh^2(p)\sinh(p)}.$$
    Differentiating \eqref{eq-f''} with respect to $u_2$, by \eqref{eq-p1} and  \eqref{eq-q1}, we have 
    \begin{equation*}\label{eq-g'''}
    0=\big(A_p \cosh^3(p)-A_q\cosh^3(q)\big)g'(u_2)+\big(B_p \cosh^3(p)-B_q\cosh^3(q)\big)g'(u_2)g''(u_2)+Bg'''(u_2), 
    \end{equation*}
    which implies that 
    $$\frac{g'''(u_2)}{g'(u_2)}=\widetilde{A}+\widetilde{B}\,g''(u_2),$$
    where $\widetilde{A}=\widetilde{A}(p,q)$ and $\widetilde{B}=\widetilde{B}(p,q)$ are rational functions of $\{\cosh(p), \sinh(p), \cosh(q), \sinh(q)\}$. 
    Differentiating this equation with respect to $u_1$, we derive that 
    $$0=\big(\widetilde{A}_p \cosh^3(p)+\widetilde{A}_q\cosh^3(q)\big)f'(u_1)+\big(\widetilde{B}_p \cosh^3(p)+\widetilde{B}_q\cosh^3(q)\big)f'(u_1)g''(u_2).$$
    It follows that $g''(u_2)=C(p,q)$ is a rational function of $\{\cosh(p), \sinh(p), \cosh(q), \sinh(q)\}$. By taking the derivative of $g''(u_2)$ with respect to $u_1$, we obtain that 
    \begin{equation}
    \label{eq-Cpq}
    C_p(p,q)\cosh^3(p)+C_q(p,q)\cosh^3(q)=0.
    \end{equation}
    Since $C_p(p,q)$ and $C_q(p,q)$ are also rational functions of $\{\cosh(p), \sinh(p), \cosh(q), \sinh(q)\}$, there exists a polynomial $Q(y_1, y_2, y_3, y_4)$ such that 
    $$Q\big(\cosh(p), \sinh(p), \cosh(q), \sinh(q)\big)=0.$$
    Expand the polynomial $Q\big(\cosh(p), \sinh(p), \cosh(q), \sinh(q)\big)$ in a series of $\sinh(np)$ and $\cosh(mp)$, with coefficients given by polynomials in $\sinh(q)$ and $\cosh(q)$. Since $p$ and $q$ are independent functions by \eqref{eq-p1} and \eqref{eq-q1}, all coefficients in this series must vanish. This leads to the conclusion that $p$ is constant. Therefore,  $f'(u_1)=0$, which is a contradiction.

    If $f'(u_1)=0, g'(u_2)\neq0$, then we regard \eqref{eq-pqs} as the following system of linear equations, 
    \begin{equation}\label{eq-fgl}
    \begin{split}
        3  g'(u_2)^2 \cosh^5(p) \sinh(p)+g''(u_2)\cosh^3(p)=-4\sinh(p)\cosh(q), \\
        3 g'(u_2)^2 \cosh^5(q) \sinh(q)-g''(u_2)\cosh^3(q)=-4\sinh(q)\cosh(p).
        \end{split}
    \end{equation}
    Its coefficient matrix also has a non-vanishing determinant given by   $$-3\cosh^3(p)\cosh^3(q)\big(\cosh^2(q) \sinh(q)+\cosh^2(p) \sinh(p)\big)$$
    under the hypothesis $p\neq \pm q$. So we can solve from \eqref{eq-fgl} that both $g'(u_2)$ and $g''(u_2)$ are rational functions of $\{\cosh(p), \sinh(p),\cosh(q), \sinh(q)\}$. Hence there also exists a polynomial $Q(y_1, y_2, y_3, y_4)$ such that 
    $$Q\big(\cosh(p), \sinh(p), \cosh(q), \sinh(q)\big)=0.$$ In this case, it follows from \eqref{eq-p1} and \eqref{eq-q1} that $p$ can be treated as a function of $q$ satisfying 
\[
\frac{dp}{dq} = -\frac{\cosh^3 p}{\cosh^3 q}.
\]
Consequently, regarding $Q\big(\cosh(p), \sinh(p), \cosh(q), \sinh(q)\big)$ as a function of $q$ alone, we deduce that $q$ is a constant. This implies $g'(u_2) = 0$, leading to a contradiction.

    A similar argument applies to the case $f'(u_1) \neq 0$ and $g'(u_2) = 0$, again yielding a contradiction.

    Therefore, we establish the claim that $p=\pm q$, which means either $v=0$ or $w=0$. It then follows from Proposition 4 in \cite{Urbano} that $x$ is non-full in $\mathbb{S}^2 \times \mathbb{S}^2$. This means it lies in a totally geodesic hypersurface of the ambient space, which, up to an isometry, is an open subset of $\mathbb{S}^2 \times \mathbb{S}^1$.    
\end{proof}
\begin{remark}\label{rk-mini}If $w=0$, then it follows from \eqref{eq-vw} that $v$ only depends on the variable $u_1$, and satisfies 
\begin{equation}\label{eq-ellip}
\frac{d^2 v}{d u_1 ^2}=-2\sinh(2v).
\end{equation}
The solution to this ordinary differential equation can be expressed using an elliptic function, 
\[
v = \operatorname{arcsinh}\Bigl( k \; \mathrm{sn}\bigl( 2u_1 + \delta,\; i k 
\bigr) \Bigr),
\]
where $\delta$ and $k>0$ are integration constants, and $\mathrm{sn}(\,\cdot,\,ik)$ denotes the Jacobian elliptic sine function with  modulus $i k$.  
Applying the classical Jacobi imaginary modulus transformation, the solution also admits the equivalent form 
\[
v = \operatorname{arcsinh}\Bigl( \frac{k}{\sqrt{1+k^2}} \; \mathrm{sn}\bigl( \sqrt{1+k^2}\,(2u_1 + \delta),\; \frac{k}{\sqrt{1+k^2}} \bigr) \Bigr).\]

Geometrically, given a solution of~\eqref{eq-ellip}, one first obtains a minimal surface $\psi:\mathbb{C}\rightarrow \mathbb{S}^3$ with induced metric $e^{2v}|dz|^2$ 
and Hopf differential $\frac i 2 dz\otimes dz$. 
Such a surface should be necessarily homogeneous or of cohomogeneity one and hence belongs to the $T_{m,k,a}$ family constructed and classified by Hsiang and Lawson in \cite{Hsiang-Lawson}. This family contains infinitely many closed minimal surfaces, whose topology is either a torus or a Klein bottle. Denote by $N$ the Gauss map of such a surface. Then the map \emph{(}see the proof of Corollary $1$ in \cite{Urbano}\emph{)}
\[
\phi(z)=(V_\psi, e^{2i u_2}):\mathbb{C}\rightarrow \mathbb{S}^2\times\mathbb{S}^1
\]
defines a minimal surface in $\mathbb{S}^2\times\mathbb{S}^1$, where
\[
V_\psi=\frac{1}{\sqrt{2}}(-2ie^{-2v}\psi_{z}\wedge\psi_{\overline{z}}+\psi\wedge N):\mathbb{C}\rightarrow \mathbb{S}^2\subset \Lambda^2\mathbb{R}^4. 
\]
Apart from the obvious slices, these surfaces constitute all minimal-Willmore surfaces in $\mathbb{S}^2\times\mathbb{S}^1$. 
\end{remark}

The classification stated in Theorem~\ref{thm1} (see the introduction) follows from Proposition~\ref{pro-mw} and Theorem~\ref{thm-mw}, together with the analyticity of minimal surfaces.

\section{Willmore surfaces of product type in $\mathbb{S}^2\times \mathbb{S}^2$}\label{sec5}
In this section, we present a classification of all Willmore surfaces of product type in $\mathbb{S}^2\times \mathbb{S}^2$. 

Consider a product surface $x = \big(x_1(u_1), x_2(u_2)\big)$ in $\mathbb{S}^2 \times \mathbb{S}^2$, with $u_1$ and $u_2$ as the arc-length parameters of $x_1$ and $x_2$, and with unit normals $n_1(u_1)=J x_1'(u_1)$ and $n_2(u_2)=- J x_2'(u_2)$, respectively. It is obvious that the tangent bundle and the normal bundle of $x$ are both flat  (i.e., $\sigma=\rho=0$). Set
$$ z\triangleq u_1+iu_2,~~~
 ~~~\xi\triangleq\frac{1}{\sqrt{2}}\big(-i n_1(u_1), n_2(u_2)\big).
$$
Then $\{x_z, x_{\bar z}, \xi, \bar{\xi}\}$ 
forms a moving frame along $x$. Direct calculation yields the fundamental data as follows, 
\begin{equation}\label{eq-C1C2}C_1=C_2=0,~~~\gamma_1=\frac{i}{\sqrt{2}},~~~\gamma_2=\frac{-i}{\sqrt{2}},~~~H=-2\sqrt{2}i{f_1}=-2\sqrt{2}i\bar{f_2},
\end{equation}
\begin{equation}\label{eq-f1f2H}
f_1=\frac{1}{4\sqrt{2}}\big(i{k}_1(u_1)-{k}_2(u_2)\big),~~~f_2=\frac{1}{4\sqrt{2}}\big(-i{k}_1(u_1)-{k}_2(u_2)\big),
\end{equation}
where $k_1(u_1)$ and $k_2(u_2)$ are the geodesic curvature of $x_1$ and $x_2$ respectively. 
\begin{theorem}
    The products of an elastic  curve in $\mathbb{S}^2$ and a great circle exhaust all Willmore surfaces of product type in $\mathbb{S}^2\times\mathbb{S}^2$. 
\end{theorem}
\begin{proof}
    Let $x=\big(x_1(u_1), x_2(u_2)\big)$ be a product surface in $\mathbb{S}^2\times\mathbb{S}^2$. 
    Substituting \eqref{eq-C1C2}  into the Willmore equation \eqref{eq-Willmore}, it follows that $x$ is Willmore if and only if  $$2({f_2})_{zz}+2(f_2)_{\bar{z}\bar{z}}-8{f_2}(f_2^2+\bar{f_2}^2)-\bar{f_2}=0.$$
    From \eqref{eq-f1f2H}, the above equation reduces to the following system of ordinary differential equations.
    \begin{align}
    2k_1''(u_1) + k_1(u_1) \left( k_1(u_1)^2 - k_2(u_2)^2 \right) + 2 k_1(u_1) &= 0, \label{eq-ode1} \\
    2 k_2''(u_2) +  k_2(u_2) \left(k_2(u_2)^2-k_1(u_1)^2\right) + 2k_2(u_2) &= 0. \label{eq-ode2}
\end{align}

If $k_1(u_1)=0$, then \eqref{eq-ode2} reduces to
$$2 k_2''(u_2) + k_2(u_2)^3+ 2k_2(u_2)=0,$$
the standard elastic equation. Therefore, $x_1(u_1)$ is a great circle and $x_2(u_2)$ is an elastic curve in $\mathbb{S}^2$. 

For the case of $k_1(u_1)\neq0$,  since $k_2(u_2)$ does not depend on the variable $u_1$, we deduce from \eqref{eq-ode1} that $k_2(u_2)=c$ is a constant. It follows from \eqref{eq-ode2} that $c$ satisfies 
$$c^3+\big(2-k_1(u_1)^2\big)c=0,$$
    which implies $c=0$ or $k_1(u_1)^2=2+c^2$. 
    We claim that the latter case is impossible. In fact, substituting $k_1(u_1)^2=2+c^2$ into \eqref{eq-ode1} yields a contradiction  
    $$4=2+c^2-c^2+2=0.$$
    Therefore, we derive that  $k_2(u_2)=0$ and $k_1(u_1)$ satisfies the standard elastic equation 
    $$2k_1''(u_1) + k_1(u_1)^3 + 2 k_1(u_1) = 0.$$
    Hence, $x_1(u_1)$ is an elastic curve in $\mathbb{S}^2$ and $x_2(u_2)$ is a great circle.   
\end{proof}
It follows immediately from \eqref{eq-C1C2} that every product surface in $\mathbb{S}^2\times\mathbb{S}^2$ is Lagrangian with respect to both complex structures $J_1$ and $J_2$. Conversely, Castro and Urbano \cite{Castro-Urbano} showed that any surface Lagrangian with respect to both \(J_1\) and \(J_2\) is necessarily a product surface. 
\begin{corollary}
If a Willmore surface  $x:\Sigma\rightarrow \mathbb{S}^2\times\mathbb{S}^2$ is Lagrangian for both complex structures $J_1$ and $J_2$, then it is the product of an elastic curve in $\mathbb{S}^2$ and a great circle. 
\end{corollary}

\textbf{Acknowledgement:}
This work was supported by NSFC No. 12171473 and 11831005, as well as the Fundamental Research Funds for the Central Universities. 
The second author was also supported by JSPS  KAKENHI Grant Number JP22K03304.

\noindent{Xiaoling Chai, Zhenxiao Xie}\\
\noindent{\small \em School of Mathematical Sciences, Beihang University, Beijing 102206, China.}\\
\noindent{\em Email: xlchai2000@163.com,~xiezhenxiao@buaa.edu.cn}\\

\noindent{Shimpei Kobayashi}\\
\noindent{\small \em Department of Mathematics, Hokkaido University, 
 Sapporo, 060-0810, Japan.}\\
\noindent{\em Email: shimpei@math.sci.hokudai.ac.jp}\\
 
 \noindent{Changping Wang}\\
\noindent{\small \em School of Mathematics and Statistics, FJKLMAA, Fujian Normal University, Fuzhou 350117, China.}\\
\noindent{\em Email: cpwang@fjnu.edu.cn}\\



\begin{thebibliography}{99}
\bibitem{Alexakis}
S. Alexakis, R. Mazzeo, {\em Renormalized area and properly embedded minimal surfaces in hyperbolic 3-manifolds}, Commun. Math. Phys. 297: 621-651 (2010).
\bibitem{WD}
D. Brander, P. Wang, {\em On the Bj\"orling problem for Willmore surfaces}, J. Differ. Geom. 108: 411-457 (2018). 

\bibitem{Bryant}
R. Bryant, {\em A duality theorem for Willmore surfaces}, J. Differ. Geom. 20: 23-53 (1984). 

\bibitem{Castro-Urbano} 
I. Castro, F. Urbano, {\em Minimal Lagrangian surfaces in $\mathbb{S}^2\times\mathbb{S}^2$}, Commun. Anal. Geom. 15: 217-248 (2007).

\bibitem{Chen}
B.Y. Chen, T. Nagano, 
{\em Totally geodesic submanifolds of symmetric spaces I}, Duke Math. J. 44: 745-755 (1997). 

\bibitem{Ejiri}
E. Ejiri, {\em Willmore surfaces with a duality in $\mathbb{S}^N(1)$}, Proc. London Math. Soc. 57: 383-416 (1988). 

\bibitem{Graham}
C.R. Graham, E. Witten, {\em Conformal anomaly of submanifold observables in AdS/CFT correspondence}, Nuclear Physics B, 546: 52-64 (1999).

\bibitem{Hu-Li}
Z.J. Hu, H.Z. Li, {\em Willmore submanifolds in a Riemannian manifold}, Contemporary geometry and related topics, 251-275 (2004). 

\bibitem{Hsiang-Lawson} 
W.-Y. Hsiang, H.B. Lawson, {\em Minimal submanifolds of low cohomogeneity}, J. Differ. Geom. 5: 1–38 (1971). 

\bibitem{Ikoma}
N. Ikoma, A. Malchiodi, A. Mondino, {\em Embedded area-constrained Willmore tori of small area in Riemannian three-manifolds I: minimization}, Proc. London Math. Soc. 115: 502-544 (2017). 

\bibitem{MWW}
X. Ma, C.P. Wang, P. Wang, {\em Classification of Willmore two-spheres in the 5-dimensional sphere}, J. Differ. Geom. 106: 245-281 (2017). 

\bibitem{Marques}
F.C. Marques, A. Neves, {\em Min-max theory and the Willmore conjecture}, Ann. of Math. 179: 683-782 (2014). 

\bibitem{Marques1}
F.C. Marques, A. Neves, {\em The Willmore Conjecture}, Jahresber. Dtsch. Math. Ver. 116: 201-222 (2014). 
\bibitem{Michelat}
A. Michelat, A. Mondino, {\em Quantization of the Willmore energy in Riemannian manifolds}, Adv. Math. 489: 110789 (2026).

\bibitem{Modino}
 A. Mondino, {\em The conformal Willmore functional: a perturbative approach}, J. Geom. Anal. 23: 764-811 (2013).
 
\bibitem{Mondino}
A. Mondino, T. Rivi\'ere, {\em Willmore spheres in compact Riemannian manifolds}, Adv. Math. 232: 608-676 (2013).

\bibitem{Montiel-Urbano} S. Montiel, F. Urbano, {\em A Willmore functional for compact surfaces in the complex projective plane},
J. reine angew. Math. 546: 139-154 (2002). 


\bibitem{Pedit}
F.J. Pedit, T.J. Willmore, {\em Conformal Geometry}, Atti Sem. Mat. Fis. Univ. Modena XXXVI: 237-245 (1988).

\bibitem{Urbano}
F. Torralbo, F. Urbano, {\em Minimal Surfaces in $\mathbb{S}^2\times\mathbb{S}^2$}, J. Geom. Anal. 25: 1132-1156 (2015).

\bibitem{Torralbo} F. Torralbo, F. Urbano,  {\em Surfaces with parallel mean curvature vector in $\mathbb{S}^2\times\mathbb{S}^2$ and $\mathbb{H}^2\times\mathbb{H}^2$}, Trans. Am. Math. Soc. 364: 785-813 (2012)


\bibitem{Wang-Xie}
C.P. Wang, Z.X. Xie, {\em Willmore surfaces in 4-dimensional conformal manifolds}, preprint, arXiv:2306.00846v2.

\bibitem{Willmore}
T.J. Willmore, {\em Note on embedded surfaces}, An. Sti. Univ. ”Al. I. Cuza” Iasi, 11: 493-496 (1965).
\end{thebibliography}
\end{document}